\documentclass[12pt]{article}

\usepackage{hyperref}

\usepackage{graphicx}      
\usepackage{amsmath,amssymb,amsfonts}
\usepackage{enumitem}

\newtheorem{theorem}{Theorem}

\newtheorem{lemma}[theorem]{Lemma}
\newtheorem{definition}{Definition}
\newtheorem{remark}{Remark}

\usepackage{amsmath,amssymb}

\usepackage{graphics} 
  \usepackage{epsfig}
\usepackage{lipsum}
\usepackage{amsfonts}
\usepackage{amssymb}
\usepackage{graphicx}
\usepackage{epstopdf}
\usepackage{graphicx}
\usepackage{enumitem}
\usepackage{tikz}
\usepackage{booktabs}
\usepackage{float}
\usetikzlibrary{positioning,arrows.meta,shadows,fit}

\usepackage{subfigure}
\ifpdf
  \DeclareGraphicsExtensions{.eps,.pdf,.png,.jpg}
\else
  \DeclareGraphicsExtensions{.eps}
\fi

\usepackage{enumitem}

\definecolor{customcolor}{RGB}{32,178,170}
\colorlet{coverlinecolor}{customcolor}

\usepackage{tikz}
\usetikzlibrary{positioning, arrows.meta, shapes, fit, shapes.misc, calc}

\DeclareMathOperator{\diag}{diag}
\DeclareMathOperator{\bfa}{\mathbf{a}}
\DeclareMathOperator{\PP}{\mathbf{P}}
\DeclareMathOperator{\barP}{\bar{\mathbf{P}}}
\usepackage{amsmath,amssymb,mathtools}
\usepackage{mathrsfs}
\usepackage{algorithm}
\usepackage[noend]{algpseudocode}

\begin{document}

\title{Distributed recursive binary identification under tampering and non-persistent excitation}

\author{
Jian Guo\footnote{Department of Applied Mathematics, 
The Hong Kong Polytechnic University, Kowloon, Hong Kong 
(\texttt{jiguo@polyu.edu.hk}).}
\ and\
Ji-Feng Zhang\footnote{School of Automation and Electrical Engineering, 
Zhongyuan University of Technology, Zhengzhou 450007, China, 
and the State Key Laboratory of Mathematical Sciences, 
Academy of Mathematics and Systems Science, 
Chinese Academy of Sciences, Beijing 100190, China 
(\texttt{jif@iss.ac.cn}).}
}

\maketitle

\begin{abstract} \noindent
In this paper, we consider distributed parameter estimation with binary observations under measurement-side tampering: each node observes a thresholded output whose label may be flipped and exchanges information over a communication graph. We develop a distributed recursive projection algorithm based on the diffusion strategy. Without imposing independence, stationarity, or Gaussian assumptions, we establish almost sure upper bounds of both the accumulated regrets of the adaptive predictor and the distributed estimation error.
Under a mild cooperative excitation condition, all nodes' estimate are consistent, even when each node is individually non-exciting. Simulations on a jointly exciting network corroborate the theory and show that the proposed algorithm converges, whereas non-cooperative and tampering-unaware baselines do not.
\end{abstract}

\section{Introduction}

Distributed estimation over networks has attracted sustained attention in recent years.
In many applications, a collection of spatially distributed agents aims to estimate a common
parameter vector by exchanging information over a communication graph.
To reduce computational and communication burdens and to enhance robustness,
incremental, consensus, and diffusion strategies have been proposed and extensively analyzed; see, e.g., \cite{cattivelli2009diffusion,chen2012diffusion,cattivelli2010diffusion,alghunaim2024local,zhang2026dive}.
Among these schemes, diffusion-type algorithms are particularly appealing due to their simple
recursions and favorable stability properties.
Classical results typically assume accurate local measurements and reliable communication links.
Under these assumptions, a broad class of distributed least-squares and stochastic approximation
methods has been shown to achieve asymptotically convergence, even under non-persistent excitation; see \cite{xie2018analysis,xie2020convergence,gan2023distributed} and references therein.

Two practical issues challenge these idealized settings.
The first is \emph{measurement and communication constraints}.
In many sensor network applications, sensing units are low-cost and battery-powered, with strict limits
on sampling precision, communication bandwidth, and storage.
Consequently, it is often infeasible to transmit high-resolution analog or multi-bit measurements.
This has motivated the study of distributed estimation with quantization.
Existing works can be roughly divided into two classes.
The first class considers accurate local measurements but quantized inter-agent communication
for exchanging intermediate estimates, e.g., within consensus or diffusion frameworks 
\cite{kar2012distributed,michelusi2022finite,zhu2018mitigating,iakovidou2022s}.
The second class considers quantized observations at the sensing stage while allowing accurate
exchange of processed information among agents; see, for instance,
distributed algorithms based on normalized least mean square or related quasi-Newton recursions for quantized data \cite{wang2021distributed,lu2025consensus,wang2026distributed}.
These studies demonstrate that, with suitable algorithm design, it is possible to approach
centralized performance despite coarse quantization.
However, most existing results rely on multi-bit uniform quantizers and do not address
secure operation under adversarial disturbances.

The second issue is \emph{security}.
As distributed estimators rely on repeated local updates and information exchange,
they are inherently vulnerable to malicious attacks.
Typical threats include falsified sensor readings \cite{pasqualetti2013attack,an2019distributed}, manipulated intermediate estimates \cite{su2019finite,an2021byzantine}, corrupted communication links that inject misleading data into the network \cite{de2015input}.
Such behaviors are naturally modeled in the framework of Byzantine adversaries.
To mitigate these risks, a number of secure or robust distributed estimation and
diffusion strategies have been proposed, including schemes with attack detection \cite{chen2018resilient,he2021secure},
innovation saturation \cite{fang2018robustifying,fang2021robust}, robust combination rules \cite{yu2022robust,han2025masked}, or game-theoretic defenses \cite{wu2020zero,paarporn2024strategically}. Related efforts also study privacy-preserving state estimation \cite{guo2025state} and Dos attack \cite{su2025event}.
These works provide important tools for resilient distributed estimation.
Nonetheless, most of them assume real-valued measurements or multi-bit messages and do not
explicitly address the joint effect of severe quantization and Byzantine-type attacks,
especially in the extreme case of binary observations.

In practice, binary sensors are widely deployed due to their low cost, simple hardware,
and low power consumption.
Examples include threshold-type detectors in structural health monitoring \cite{farrar2012structural}, industrial
alarm systems \cite{farrar2007introduction}, intrusion detection \cite{khraisat2019survey}, and environmental surveillance \cite{banerjee2008fault}.
In such systems, each node reports only a one-bit observation, e.g., ``above/below threshold''.
At the same time, the binary structure makes these systems particularly exposed:
bit-flip attacks or corrupted decision bits are easy to implement and difficult to detect.
Moreover, in many applications, a single sensor may not provide sufficient excitation
for consistent identification, so reliable estimation relies essentially on cooperation.
These observations raise a natural but unexplored question:
\emph{how to design secure distributed estimation algorithms under binary observations
subject to bit-flip, a Byzantine-type attacks, possibly in non-PE and non-i.i.d. observations?}

This work addresses this question within a diffusion strategy.
We consider a network of agents observing a common linear model through binary sensors.
Each agent's one-bit output is subject to random flips before being used in the local update.
Agents perform local stochastic recursions based on these corrupted binary data and then
combine their intermediate estimates with neighbors via a diffusion strategy.
The goal is to recover the true parameter vector cooperatively, even when individual agents
lack sufficient excitation and when the binary observations of the agent are under Byzantine-type attacks.

The contributions of this paper are twofold.
\begin{enumerate}
  \item We propose a distributed recursive projection algorithm tailored to
  binary-valued observations with measurement-side tampering. The adaptation step performs a projected update that accounts for bit flips and the combination step mixes estimates and information matrices with neighbors. Taken together, This yields a distributed
  diffusion recursion that simultaneously captures binary-valued observation, random flips, and
 cooperation in the network.

  \item We develop a convergence theory for the proposed distributed algorithm.
  Under general stochastic conditions, without imposing independence or stationarity on the system signals, we derive almost-sure upper bounds of both the accumulated regrets of the adaptive predictor and the distributed estimation error of all nodewise estimates despite binary sensing and measurement-side
  tampering. A cooperative excitation condition is proposed for the tampered binary
  observations and shown to be an extension of the single-node:
  it allows the network to identify the true parameter even when no individual agent
  can do so on its own.
\end{enumerate}

The paper is organized as follows. Section 2 introduces the model and presents the distributed identification algorithm. Section 3 establishes the convergence theory of the proposed algorithm. Section 4 offers numerical simulations. Section 5 concludes the paper and outlines future directions.

\subsection{Preliminaries}

{\bf Notation.}
Let $\mathbb{S}^n$ be the set of $n\times n$ real symmetric matrices and $\mathbb{S}_{++}^n$ be the cone of positive definite matrices. For $A,B\in\mathbb{S}^n$, $A\succeq B$ means that $A-B$ is positive semidefinite. For $A\in\mathbb{R}^{m\times n}$ and $B\in\mathbb{R}^{p\times q}$, the Kronecker product is $A\otimes B\in\mathbb{R}^{mp\times nq}$. For a square matrix $M\in\mathbb{R}^{n\times n}$, the determinant is $|M|=\det(M)$. The indicator function is $I\{\text{statement}\}=1$ if the statement is true and $I\{\text{statement}\}=0$ otherwise. The Euclidean norm of a vector $x\in\mathbb{R}^n$ is $\|x\|=\sqrt{x^\top x}$. The Euclidean (spectral) norm of a matrix $A\in\mathbb{R}^{m\times n}$ is $\|A\|=\sqrt{\lambda_{\max}(A^\top A)}$, where $\lambda_{\max}(\cdot)$ denotes the largest eigenvalue of a symmetric matrix. Correspondingly, $\lambda_{\min}(\cdot)$ denotes the smallest eigenvalue of a symmetric matrix. Let $[n]=\{1,2,\dots,n\}$.

For vectors $\{x_i\}$, define 
$
\operatorname{col}\{x_1,\dots,x_n\}:=\begin{bmatrix}x_1^\top & \cdots & x_n^\top\end{bmatrix}^\top .
$
For a vector $x=\operatorname{col}\{x_1,\dots,x_n\}$, write 
$\operatorname{col}_i\{g(x_i)\}:=\operatorname{col}\{g(x_1),\dots,g(x_n)\}$ and  
$\operatorname{diag}\{\cdot\}$ forms a block-diagonal matrix.  For a block-diagonal matrix $H\!=\!\operatorname{diag}\{H_1,\dots,H_n\}$, set
$\boldsymbol{\Pi}_{H}\{\zeta\}:=\operatorname{col}_i\big(\Pi_{H_i}\{\zeta_i\}\big)$ for 
$\zeta=\operatorname{col}\{\zeta_1,\dots,\zeta_n\}$. Finally, $\operatorname{vec}\!\big(\operatorname{diag}(H_1,\ldots,H_n)\big):=
\begin{bmatrix}
H_1\\[-1pt]\vdots\\[-1pt] H_n
\end{bmatrix}.$ 

{\bf Communication graph.}
The communication network is modeled by an undirected weighted graph $G=(V,E,A)$, where $V=[n]$ is the set of nodes, $E\subseteq V\times V$ is the undirected edge set, and $A=[a_{ij}]\in\mathbb{R}^{n\times n}$ is the weighted adjacency matrix. We assume $a_{ij}\ge 0$, $a_{ij}=a_{ji}$ for all $i,j$, $\sum_{j=1}^n a_{ij}=1$ for all $i$, and $a_{ij}>0$ only if $(i,j)\in E$ (otherwise $a_{ij}=0$). Under these assumptions $A$ is symmetric and row-stochastic, hence $A$ is doubly stochastic. The neighbor set of node $i$ is $N_i=\{\,j\in V:\ (i,j)\in E\,\}=\{\,j\in V:\ a_{ij}>0\,\}$.
A path of length $\ell\in\mathbb{N}$ is a sequence $(i_0,i_1,\ldots,i_\ell)$ with $(i_{k-1},i_k)\in E$ for all $k=1,\ldots,\ell$. The graph distance $\mathrm{dist}_G(i,j)$ is the length of a shortest path in $G$ connecting nodes $i$ and $j$ (with $\mathrm{dist}_G(i,i)=0$ and $\mathrm{dist}_G(i,j)=\infty$ if $i$ and $j$ lie in different components). The graph $G$ is connected if $\mathrm{dist}_G(i,j)<\infty$ for all $i,j\in V$. The diameter of $G$ is $D(G)=\max_{i,j\in V}\mathrm{dist}_G(i,j)$.

\section{Model Formulation}

\subsection{Model Description Under Byzantine Attack}

We consider a network with nodes $V=[n]$ and communication graph $G=(V,E,A)$. 
At each time $k\ge 0$, node $i\in V$ forms an {observable} regressor $\varphi_{k,i}\in\mathbb{R}^p$ that may collect the current input and a finite number of past inputs, and the plant output is
\begin{equation}\label{eq:plant}
y_{k+1,i}=\varphi_{k,i}^\top\theta+w_{k+1,i},\qquad i\in[n],
\end{equation}
where $\theta\in\mathbb{R}^p$ is an unknown time-invariant parameter and $\{w_{k+1,i}\}$ is a noise sequence.
Denote by $\{\mathcal{F}_k\}_{k\ge 0}$ the natural filtration in this network setting:
$$
\begin{aligned}
\mathcal{F}_k
=&\sigma\!\Big(\,\{\,y_{j,i},\ u_{j,i},\ w_{j,i},\ w'_{j,i}\,:\ i\in[n],\ 1\le j\le k\,\}\\
&\cup\{\,u_{0,i}: i\in[n]\,\}\Big),\quad k\ge 0,
\end{aligned}
$$
where $\{u_{j,i}\}$ and $\{w'_{j,i}\}$ denote, respectively, the system inputs and a possible exogenous input sequence at node $i$.

\noindent\textbf{Byzantine attacks in the exact-output setting.}
In secure distributed estimation, adversarial actions are often grouped into {measurement attacks} and {communication-link attacks}. In this work we focus on {measurement attacks}. If $y_{k+1,i}$ were directly observable, a measurement-side Byzantine adversary alters the sensor reading at the sensing device. Typical models include
\begin{align}
&\text{(i) arbitrary replacement:}\ y_{k+1,i}\ \mapsto\ \tilde y_{k+1,i},\nonumber\\
&\text{(ii) multiplicative scaling:}\ y_{k+1,i}\ \mapsto\ \alpha_{k,i}\,y_{k+1,i},\ \alpha_{k,i}>0,\nonumber
\end{align}
see, e.g., \cite{yin2018byzantine,an2021byzantine}.

\noindent\textbf{Binary sensing.}
In our setting, the real-valued output is not directly available. Each node $i$ has a fixed threshold $C\in\mathbb{R}$ and produces the binary signal
\begin{equation}\label{eq:binary}
s^{0}_{k+1,i}=I\{\,y_{k+1,i}\le C\,\}.
\end{equation}

\noindent\textbf{Byzantine attacks in the binary setting.}
We assume the adversary tampers the measurement-side binary output produced by the local comparator (e.g., biasing the comparator or spoofing its digital output). The bit entering the estimator at node $i$ is modeled via node-dependent flipping probabilities
\begin{equation}\label{eq:flip}
\left\{
\begin{aligned}
\Pr\{\,s_{k,i}=0\mid s^{0}_{k,i}=1\,\}&=p_i,\\
\Pr\{\,s_{k,i}=1\mid s^{0}_{k,i}=0\,\}&=q_i,
\end{aligned}
\right.
\quad p_i,q_i\in[0,1),\  i\in[n].
\end{equation}
Thus $(p_i,q_i)$ quantify node-side misclassification induced by Byzantine tampering. The information flow for distributed binary identification under Byzantine tampering is shown in Fig.~\ref{fig:attack}.

\medskip
\noindent\textbf{Identification objective.}
Given the observable regressors $\{\varphi_{k,i}\}$ and the possibly tampered binary observations $\{s_{k+1,i}\}$ generated by \eqref{eq:plant}-\eqref{eq:flip}, our aim is to design a \emph{distributed} recursive algorithm to jointly estimate $\theta$.

\begin{figure}[t]
\centering
\resizebox{1\linewidth}{!}{
\begin{tikzpicture}[
    node distance=0.5cm and 0.9cm,
    >={Triangle[length=4pt,width=4pt]},
    every node/.style={font=\small},
    block/.style={draw, rectangle, rounded corners=2pt, fill=blue!5, minimum height=8mm, minimum width=20mm, align=center},
    cloud/.style={shape=cloud, draw=red!60!black, fill=red!10, cloud puffs=11, cloud ignores aspect,
                  minimum width=18mm, minimum height=7mm, align=center},
    net/.style={draw, rectangle, rounded corners=2pt, fill=green!5, minimum height=16mm, minimum width=28mm, align=center},
    point/.style={circle, fill, inner sep=0pt, minimum size=2pt}
]


\def\xB{4.14}   
\def\xC{8.28}   
\def\xD{12.42}  
\def\xE{15.41}  
\def\xF{20}  

\def\yTop{2.5}
\def\yMid{0}
\def\yBot{-2.5}

\node (network) at (\xF,0) [net] {Communication /\\ Fusion center};

\node (plant1) [block] at (\xB,\yTop) {System Plant\\(Node 1)};
\node (sens1)  [block, right=1.3cm of plant1] {Binary\\Sensor};
\node (atk1)   [cloud, right=1.3cm of sens1] {Byzantine\\tampering};
\node (est1)   [block, right=1.3cm of atk1] {Local Estimator \\$\psi_{k+1,1}$};

\coordinate (u1) at ($(plant1.west)+(-8mm,0)$);
\draw[->, thick] (u1) -- node[midway, above] {$u_{k,1}$} (plant1.west);

\draw[->, thick] (plant1) -- node[midway, above] {$y_{k+1,1}$} (sens1);
\draw[->, thick] (sens1) -- node[midway, above] {$s^{0}_{k+1,1}$} (atk1);
\draw[->, thick] (atk1)  -- node[midway, above] {$s_{k+1,1}$} (est1);

\node (plant2) [block] at (\xB,\yMid) {System Plant\\(Node 2)};
\node (sens2)  [block, right=1.3cm of plant2] {Binary\\Sensor};
\node (est2)   [block, right=3cm of sens2] {Local Estimator \\$\psi_{k+1,2}$};

\coordinate (u2) at ($(plant2.west)+(-8mm,0)$);
\draw[->, thick] (u2) -- node[midway, above] {$u_{k,2}$} (plant2.west);

\draw[->, thick] (plant2) -- node[midway, above] {$y_{k+1,2}$} (sens2);
\draw[->, thick] (sens2) -- node[midway, above] {$s_{k+1,2}=s^{0}_{k+1,2}$} (est2);

\node at (\xB, \yBot+1.4) {$\vdots$};
\node at (\xC-0.4, \yBot+1.4) {$\vdots$};
\node at (\xE, \yBot+1.4) {$\vdots$};

\node (plantn) [block] at (\xB,\yBot) {System Plant\\(Node $n$)};
\node (sensn)  [block, right=1.3cm of plantn] {Binary\\Sensor};
\node (atkn)   [cloud, right=1.3cm of sensn] {Byzantine\\tampering};
\node (estn)   [block, right=1.3cm of atkn] {Local Estimator\\$\psi_{k+1,n}$};

\coordinate (un) at ($(plantn.west)+(-8mm,0)$);
\draw[->, thick] (un) -- node[midway, above] {$u_{k,n}$} (plantn.west);

\draw[->, thick] (plantn) -- node[midway, above] {$y_{k+1,n}$} (sensn);
\draw[->, thick] (sensn) -- node[midway, above] {$s^{0}_{k+1,n}$} (atkn);
\draw[->, thick] (atkn)  -- node[midway, above] {$s_{k+1,n}$} (estn);

\coordinate (fL1) at ($(network.south west)!0.4!(network.north west)$);
\coordinate (fL2) at ($(network.south west)!0.5!(network.north west)$);
\coordinate (fL3) at ($(network.south west)!0.6!(network.north west)$);

\draw[->, thick] (est1.east) -- (fL3); 
\draw[->, thick] (est2.east) -- (fL2); 
\draw[->, thick] (estn.east) -- (fL1); 

\node (ret1)   [point, right=0.3cm of network.east] {};
\node (theta1) [block, right=0.7cm of ret1, yshift=1.7cm] {$\theta_{k+1,1}$};
\node (theta2) [block, right=0.7cm of ret1, yshift=0.0cm] {$\theta_{k+1,2}$};
\node (thetan) [block, right=0.7cm of ret1, yshift=-1.7cm] {$\theta_{k+1,n}$};
\node at ($(theta2.south)!0.4!(thetan.north)$) {$\vdots$};

\draw[-, thick] (network.east) -- (ret1);
\draw[->, thick] (ret1) |- (theta1);
\draw[->, thick] (ret1) -- (theta2);
\draw[->, thick] (ret1) |- (thetan);

\end{tikzpicture}
}
\caption{Information flow for distributed binary identification under Byzantine attacks.}
\label{fig:attack}
\end{figure}
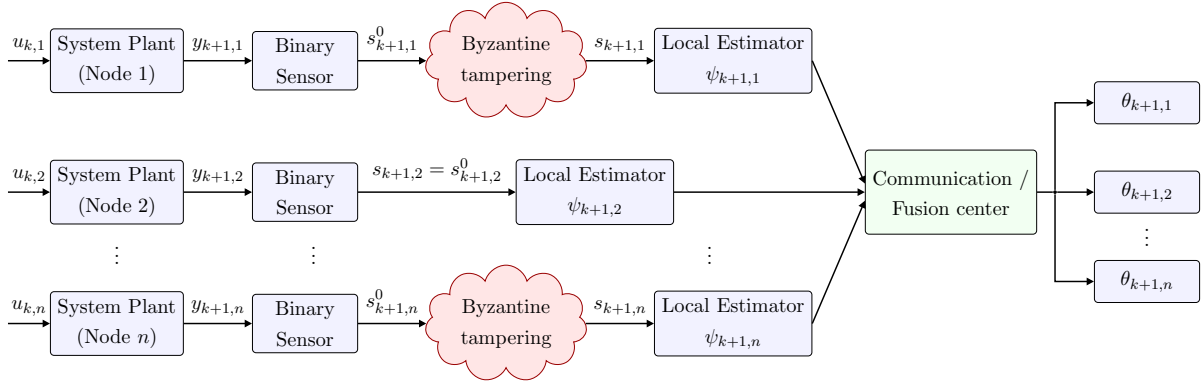

\subsection{Identifiability Analysis}

Let $F_{k,i}(\cdot)$ denote the conditional distribution function of the noise $w_{k+1,i}$ given the filtration $\mathcal{F}_k$  at node $i$.
Similar to \cite{guo2025recursivebinaryidentificationdata}, under the binary sensing (\ref{eq:binary}) and the tampering model (\ref{eq:flip}), the conditional probabilities of the measurement at node $i$ are
\begin{align}
&\Pr\{s_{k+1,i}=0\mid \mathcal{F}_k\}\\
=&\Pr\{s^0_{k+1,i}=1\mid \mathcal{F}_k\}\Pr\{s_{k+1,i}=0\mid s^0_{k+1,i}=1,\mathcal{F}_k\}\notag\\
&+\Pr\{s^0_{k+1,i}=0\mid \mathcal{F}_k\}\Pr\{s_{k+1,i}=0\mid s^0_{k+1,i}=0,\mathcal{F}_k\}\notag\\
=&p_i\,F_{k,i}\big(C-\theta^\top \varphi_{k,i}\big) + (1-q_i)\Big[1-F_{k,i}\big(C-\theta^\top \varphi_{k,i}\big)\Big]\notag\\
=&(p_i+q_i-1)\,F_{k,i}\!\big(C-\theta^\top \varphi_{k,i}\big)+1-q_i,\label{eq:multi_dist0}
\end{align}
and
\begin{align}
&\Pr\{s_{k+1,i}=1\mid \mathcal{F}_k\}
=1-\Pr\{s_{k+1,i}=0\mid \mathcal{F}_k\}\notag\\
=&\big(1-(p_i+q_i)\big)\,F_{k,i}\big(C-\theta^\top \varphi_{k,i}\big)+q_i.\label{eq:multi_dist1}
\end{align}

\medskip
\noindent\textbf{Identifiability.}
From \eqref{eq:multi_dist0}-\eqref{eq:multi_dist1}, if $p_i+q_i=1$, then the distribution of $s_{k+1,i}$ is independent of $\theta$. Thus the binary observations at node $i$ carry no information about $\theta$.
Thus, for each node $i$, a necessary condition for the local
identifiability of $\theta$ from $\{\varphi_{k,i}, s_{k+1,i}\}$ is
\begin{equation}\label{eq:ident_cond_pq}
p_i+q_i\neq 1.
\end{equation}
In particular, a node with $p_i+q_i=1$ carries no information about $\theta$
and can be discarded in the identification.

Beyond above, recovering $\theta$ from $\{\varphi_{k,i},s_{k+1,i}\}$ requires sufficient excitation of the regressors (cf.\ \cite{ljung1995system}).
The concrete excitation condition will be given later.
In this paper we first analyze the case where the node-dependent attack parameters $(p_i,q_i)$ are {known}, which allows a clean convergence analysis.
Extensions to unknown and possibly time-varying $(p_{k,i},q_{k,i})$ are left for future work.

\subsection{Assumptions under Byzantine Measurement Attacks}

With the preceding model and attack setting, we now state assumptions on the parameter set, regressors, noises, binary sensing under attacks, and the graph.

\noindent\textbf{Assumption 1 (Parameter set).}
There exists a known nonempty compact convex set $\Omega\subset\mathbb{R}^p$ such that $\theta\in\Omega$.
Let $L:=\sup_{\eta\in\Omega}\|\eta\|<\infty$.

\noindent\textbf{Assumption 2 (Regressors).}
For each $i\in[n]$ and $k\ge 0$, the regressor $\varphi_{k,i}\in\mathbb{R}^p$ is $\mathcal{F}_k$-measurable.
Moreover,
\[
\sup_{k\ge 0}\ \max_{i\in[n]}\ \|\varphi_{k,i}\| \le M < \infty\quad\text{a.s.}
\]

\noindent\textbf{Assumption 3 (Noises).}
For each $k\ge 0$ and $i\in[n]$, conditional on $\mathcal{F}_k$,
the noise $w_{k+1,i}$ has a conditional
distribution function (cdf) $F_{k,i}$ and density $f_{k,i}:=F_{k,i}'$.
Moreover, conditional on $\mathcal{F}_k$, the components
$w_{k+1,1},\ldots,w_{k+1,n}$ are independent.
In addition, there exists a constant $f_{\min}>0$ such that for all
$k\ge0$ and $i\in[n]$,
\[
\inf_{k\ge0,i\in[n]}\inf_{x\in[C-LM,C+LM]} f_{k,i}(x)\ge f_{\min}.
\]

\noindent\textbf{Assumption 4 (Graph).}
The communication graph $G=(V,E,A)$ is connected.

\begin{remark}
Assumption~1 specifies a compact convex set $\Omega$ for $\theta$, which makes projections well defined (see Algorithm \ref{algo1}), gives a finite bound $L$, and fixes the working interval $[\,C-LM,C+LM]$ so the iterates remain bounded. Assumption~2 requires $\mathcal{F}_k$-measurable regressors, ensuring causality. Assumption~3 imposes conditionally independent noises across nodes with conditional laws $F_{k,i}$ whose densities are
uniformly bounded below on the relevant interval. This prevents degeneracy near the threshold, yields a positive local slope, and supports identifiability. Assumption~4 requires a connected graph $G$, which enables distributed convergence.
\end{remark}

\subsection{Distributed Recursive Projection Algorithm with\\ Tampering}

To proceed, we recall a projection operator onto $\Omega$.

\begin{definition}[Projection w.r.t.\ $Q$]\label{def:projQ}
Let $\Omega\subset\mathbb{R}^p$ be nonempty, compact, and convex, and let $Q\in\mathbb{S}_{++}^p$. Define
\[
\Pi_{Q}(\eta)\ :=\ \arg\min_{\omega\in\Omega}\ \|\eta-\omega\|_{Q},
\ \ 
\|\eta\|_{Q}:=\sqrt{\eta^\top Q\,\eta},\ \ \eta\in\mathbb{R}^p .
\]
\end{definition}

By nonexpansiveness of metric projections, it follows
\[
\|\omega-\Pi_{Q}(\eta)\|_{Q}\ \le\ \|\omega-\eta\|_{Q},
\ \forall\,\omega\in\Omega,\ \eta\in\mathbb{R}^p .
\]

We now present an ATC (adaptation-then-combination) distributed recursive projection method that incorporates the known node-wise tampering probabilities $(p_i,q_i)$. The update uses the projection $\Pi_Q$ onto $\Omega$.

\begin{algorithm}[htbp]
\caption{ATC Distributed Recursive Projection under Binary Tampering}
\label{algo1}
\begin{algorithmic}[1]
\State \textbf{Inputs:} The (weighted) adjacency matrix $A$, threshold $C$,
conditional cdfs $\{F_{k,i}\}_{k\geq 0,\,i\in[n]}$ and densities
$\{f_{k,i}\}_{k\geq 0,\,i\in[n]}$, tampering probabilities $\{(p_i,q_i)\}_{i\in[n]}$.
  \State \textbf{Initialization:} For $i\in[n]$, choose $\theta_{0,i}\in\Omega$, $P_{0,i}\in\mathbb{S}_{++}^p$, 
  and define
 $$
  \beta_{i}=\operatorname{sign}\!\big(1-(p_i+q_i)\big)\;
  \min_{j\in[n]}\!\left\{|1-(p_j+q_j)|f_{\min}\right\}.
  $$
  \For{$k=0,1,\ldots$}
   \State \textbf{Adaptation:} For $i\in[n]$,
    \begin{align}
    a_{k,i}&=\frac{1}{1+\beta_{i}^2\,\varphi_{k,i}^\top P_{k,i}\varphi_{k,i}},\label{a_ki}\\
    \tilde{s}_{k+1,i}&=\big(1-(p_i+q_i)\big)\,F_{k,i}\big(C-\theta_{k,i}^\top\varphi_{k,i}\big)\nonumber\\
    &\quad +q_i-s_{k+1,i},\label{s_ki}\\
    \bar P_{k+1,i}&=P_{k,i}-\beta_{i}^2\,a_{k,i}\,P_{k,i}\varphi_{k,i}\varphi_{k,i}^\top P_{k,i},\label{barP_ki}\\
    \psi_{k+1,i}&=\Pi_{\bar P_{k+1,i}^{-1}}\Big(\theta_{k,i}+a_{k,i}\beta_{i}P_{k,i}\varphi_{k,i}\,\tilde{s}_{k+1,i}\Big).\label{psi_ki}
    \end{align}
    \State \textbf{Combination:} For $i\in[n]$,
    \begin{align}
    P_{k+1,i}^{-1}&=\sum_{j\in\mathcal{N}_i} a_{ij}\,\bar P_{k+1,j}^{-1},\label{P_ki}\\
    \theta_{k+1,i}&=P_{k+1,i}\sum_{j\in\mathcal{N}_i} a_{ij}\,\bar P_{k+1,j}^{-1}\,\psi_{k+1,j}.\label{theta_ki}
    \end{align}
  \EndFor
\end{algorithmic}
\end{algorithm}

Define, for each $i\in[n]$ and $k\ge 0$,
\begin{align}
&\varepsilon_{k+1,i}
\!:=\! \big(1\!-\!(p_i\!+\!q_i)\big)F_{k,i}\big(C-\varphi_{k,i}^\top\theta\big)+q_i - s_{k+1,i},\label{varepsilon}\\
&\gamma_{k,i}
\!:=\! \big(1\!-\!(p_i\!+\!q_i)\big)\Big[F_{k,i}\!\big(C\!-\!\varphi_{k,i}^\top\theta_{k,i}\big)\!-\!F_{k,i}\!\big(C\!-\!\varphi_{k,i}^\top\theta\big)\!\Big]\nonumber.
\end{align}
Then, using the residual definition \eqref{s_ki}, the adaptation step \eqref{psi_ki} can be written as
\begin{equation*}
\psi_{k+1,i}
=\Pi_{\bar P_{k+1,i}^{-1}}\!\Big(\,
\theta_{k,i}+a_{k,i}\,\beta_{i}\,P_{k,i}\varphi_{k,i}\,\big(\gamma_{k,i}+\varepsilon_{k+1,i}\big)
\Big).
\end{equation*}
Moreover, under Assumption~3, for $i\in[n]$, $\{\varepsilon_{k+1,i},\mathcal{F}_k\}$ is a martingale difference sequence:
\[
\mathbb{E}\big[\varepsilon_{k+1,i}\mid\mathcal{F}_k\big]=0,\qquad
|\varepsilon_{k+1,i}|\le 1\ \ \text{a.s.},
\]
so that for all $\beta>0$,
$
\sup_{k\ge 0}\ \mathbb{E}\!\big[\,|\varepsilon_{k+1,i}|^{\beta}\,\big|\,\mathcal{F}_k\big]\ \le\ 1.
$

\section{Main Results}

For compactness, 
we introduce stacked notations (dimensions indicated on the right):
\[
\begin{array}{ll}
Y_{k}:=\operatorname{col}\{y_{k,1},\dots,y_{k,n}\}, & n\times 1,\\[2pt]
S_{k}:=\operatorname{col}\{s_{k,1},\dots,s_{k,n}\}, & n\times 1,\\[2pt]
\Xi_{k}:=\operatorname{col}\{\varepsilon_{k,1},\dots,\varepsilon_{k,n}\}, & n\times 1,\\[2pt]
\Gamma_{k}:=\operatorname{col}\{\gamma_{k,1},\dots,\gamma_{k,n}\}, & n\times 1,\\[2pt]
B:=\operatorname{diag}\{\beta_1,\dots,\beta_n\}, & n\times n,\\[2pt]
\Phi_{k}:=\operatorname{diag}\{\varphi_{k,1},\dots,\varphi_{k,n}\}, & (pn)\times n,\\[2pt]
\Theta:=\operatorname{col}\{\theta,\dots,\theta\}, & (pn)\times 1,\\[2pt]
\Theta_{k}:=\operatorname{col}\{\theta_{k,1},\dots,\theta_{k,n}\}, & (pn)\times 1,\\[2pt]
\Psi_{k}:=\operatorname{col}\{\psi_{k,1},\dots,\psi_{k,n}\}, & (pn)\times 1,\\[2pt]
\widetilde{\Theta}_{k}:=\Theta_{k}-\Theta,\
\widetilde{\Psi}_{k}:=\Psi_{k}-\Theta,\quad & (pn)\times 1,\\[2pt]
\PP_{k}:=\operatorname{diag}\{P_{k,1},\dots,P_{k,n}\},& (pn)\times(pn),\\[2pt]
\barP_{k}:=\operatorname{diag}\{\bar P_{k,1},\dots,\bar P_{k,n}\}, & (pn)\times(pn),\\[2pt]
\bfa_{k}:=\operatorname{diag}\{a_{k,1},\dots,a_{k,n}\}, & n\times n,\\[2pt]
b_{k}:=a_{k}\otimes I_{p}, & (pn)\times(pn),\\[2pt]
\mathcal{A}:=A\otimes I_{p}, & (pn)\times(pn),\\[2pt]
\Delta:=\operatorname{diag}\{1\!-\!(p_1\!+\!q_1),\dots,1\!-\!(p_n\!+\!q_n)\},& n\times n,\\[2pt]
\mathbf{q}:=\operatorname{col}\{q_1,\dots,q_n\}, &{n}\times 1.
\end{array}
\]

From $y_{k+1,i}=\varphi_{k,i}^\top\theta+w_{k+1,i}$ for $i\in[n]$ and measurement-side flipping (\ref{eq:binary})-(\ref{eq:flip}), we have
\[
\mathbb{E}\!\left[S_{k+1}\mid \mathcal{F}_{k}\right]
=\Delta\,F_{k,i}\!\big(C\mathbf{1}-\Phi_k^\top\Theta\big)+\mathbf{q},
\]
where $F_{k,i}$ acts elementwise on vector arguments and $\mathbf{1}\in\mathbb{R}^{n}$ denotes the $n$-dimensional all-ones vector. Consequently, under Assumption~3, $\mathbb{E}[\Xi_{k+1}\mid\mathcal{F}_k]=0$ componentwise and $\|\Xi_{k+1}\|_\infty\le 1$ a.s. and thus, for any $\beta>0$,
$
\sup_{k\ge 0}\ \mathbb{E}\!\left[\|\Xi_{k+1}\|_\infty^{\beta}\,\big|\,\mathcal{F}_k\right]\le 1 .
$ In addition, denote $\widetilde\theta_{k+1,i}=\theta_{k+1,i}-\theta$ for $i\in[n]$ and $k\geq 0$ and let
$\bar{\beta}=|\beta_1|$.

\noindent\textbf{Block form of Algorithm \ref{algo1}.} From the notations above, Algorithm~\ref{algo1} can be written as
\begin{equation}\label{eq:block_ATC}
\left\{
\begin{aligned}
&\bar \PP_{k+1}
=\PP_{k}-(B\otimes I_{p})\, b_{k}\,\PP_{k}\,\Phi_{k}\Phi_{k}^\top \PP_{k}\,(B\otimes I_{p}),\\[2pt]
&\Psi_{k+1}
\!=\!\boldsymbol{\Pi}_{\bar \PP_{k+1}^{-1}}
\Big(
\Theta_{k}
\!+ \!\big((B\, a_k)\otimes I_{p}\big)\PP_{k}\,\Phi_{k}\big(\Gamma_{k}+\Xi_{k+1}\big)
\Big),\\[2pt]
&\operatorname{vec}\!\big(\PP_{k+1}^{-1}\big)
=\mathcal{A}\operatorname{vec}\!\big(\bar \PP_{k+1}^{-1}\big),\\[2pt]
&\Theta_{k+1}
=\PP_{k+1}\mathcal{A}\bar \PP_{k+1}^{-1}\ \Psi_{k+1}.
\end{aligned}
\right.
\end{equation}
Here $\boldsymbol{\Pi}_{\bar \PP_{k+1}^{-1}}$ means for any
$z=\operatorname{col}\{z_1,\dots,z_n\}$ with $z_i\in\mathbb{R}^p$,
\[
\boldsymbol{\Pi}_{\bar \PP_{k+1}^{-1}}(z)
:= \operatorname{col}\big\{\Pi_{\bar P_{k+1,1}^{-1}}(z_1),\dots,
\Pi_{\bar P_{k+1,n}^{-1}}(z_n)\big\}.
\]

\noindent\textbf{Error recursion.}
Since $\PP_{k+1}\mathcal{A}\bar \PP_{k+1}^{-1}(\mathbf{1}\otimes I_{p})=\mathbf{1}\otimes I_{p}$,
\begin{align}
\widetilde{\Theta}_{k+1}
&=\Theta_{k+1}-\Theta
=\PP_{k+1}\ \mathcal{A}\bar\PP_{k+1}^{-1}\ \big(\Psi_{k+1}-\Theta\big)\nonumber\\
&=\PP_{k+1}\ \mathcal{A}\bar\PP_{k+1}^{-1}\ \widetilde{\Psi}_{k+1}.
\label{eq:block_err_rec}
\end{align}

\noindent We next present a growth result, which shows that, despite binary sensing and measurement-side tampering, the Lyapunov energy $\widetilde{\Theta}_{t+1}^{\top}\PP_{t+1}^{-1}\widetilde{\Theta}_{t+1}$ and the cumulative predictive error $\sum_{k=0}^{t}\widetilde{\Theta}_{k}^{\top}\Phi_{k}\Phi_{k}^{\top}\widetilde{\Theta}_{k}$ admit logarithmic upper bounds under Assumptions~1-4.

\begin{theorem}[Growth bounds]\label{thm:growth}
Consider the system plant \eqref{eq:plant}-\eqref{eq:binary} with measurement-side binary tampering \eqref{eq:flip} and Algorithm~\ref{algo1}. 
Under Assumptions~1-4 and $p_i+q_i\neq 1$ for $i\in[n]$, as $t\to\infty$, it follows
\begin{align}
&\sum_{k=0}^{t}\ \widetilde{\Theta}_{k}^{\top}\,\Phi_{k}\Phi_{k}^{\top}\,\widetilde{\Theta}_{k}
= O\!\left(\log r_t\right)\quad\text{a.s.}, \label{thm1}\\
&\widetilde{\Theta}_{t+1}^{\top}\,\PP_{t+1}^{-1}\,\widetilde{\Theta}_{t+1}
= O\!\left(\log r_t\right)\quad\text{a.s.},\label{thm2}
\end{align}
with
$
r_t:=\max_{i\in[n]}\lambda_{\max}(P_{0,i})+\sum_{i=1}^{n}\sum_{k=0}^{t}\|\varphi_{k,i}\|^2 .
$
\end{theorem}

Before proving, we give two lemmas tailored to Algorithm~\ref{algo1}.

\begin{lemma}[Uniform boundedness \cite{wang2021distributed}]\label{lem:bounded}
Under Assumption~1 and Algorithm~\ref{algo1},
\[
\sup_{k\geq 0,i\in[n]}\|\psi_{k+1,i}\|\le L,\ \sup_{k\geq 0,i\in[n]}\|\theta_{k+1,i}\|\le L .
\]
\end{lemma}

\begin{lemma}[\cite{xie2020convergence}]\label{lem:P-rel}
For $k\ge 0$,
\[
\mathcal A\,\PP_{k+1}\,\mathcal A\preceq\ \bar \PP_{k+1},
\
\big|\bar \PP_{k+1}^{-1}\big|\ \le\ \big|\PP_{k+1}^{-1}\big|.
\]
\end{lemma}

{\noindent\bf Proof of Theorem \ref{thm:growth}.}
Set $V_k:=\widetilde{\Theta}_{k}^{\top}\PP_{k}^{-1}\widetilde{\Theta}_{k}$.
Using \eqref{eq:block_err_rec}
and $\mathcal{A}\PP_{k+1}\mathcal{A}\preceq \bar \PP_{k+1}$ (Lemma~\ref{lem:P-rel}), we obtain
\begin{equation}
V_{k+1}
=\widetilde{\Theta}_{k+1}^{\top}\PP_{k+1}^{-1}\widetilde{\Theta}_{k+1}
\le \widetilde{\Psi}_{k+1}^{\top}\bar \PP_{k+1}^{-1}\widetilde{\Psi}_{k+1}.
\label{eq:Vk-next}
\end{equation}
By nonexpansiveness of the projection in the $\bar P_{k+1,i}^{-1}$-metric and the adaptation in \eqref{eq:block_ATC}, we have
\begin{align}
&\widetilde{\Psi}_{k+1}^{\top}\bar \PP_{k+1}^{-1}\widetilde{\Psi}_{k+1}\nonumber\\
\le &\Big(\widetilde{\Theta}_k + ((B \bfa_k)\!\otimes\! I_p)\PP_k\Phi_k(\Gamma_k+\Xi_{k+1})\Big)^{\!\top}\!
\bar \PP_{k+1}^{-1}\nonumber\\
&\  \Big(\widetilde{\Theta}_k + ((B \bfa_k)\!\otimes\! I_p)\PP_k\Phi_k(\Gamma_k+\Xi_{k+1})\Big).
\label{eq:proj-step}
\end{align}

Combining \eqref{eq:proj-step}–\eqref{eq:Vk-next}, expanding the square, and using the block Woodbury identity in \cite{hager1989updating} implied by \eqref{eq:block_ATC},
\[
\bar \PP_{k+1}^{-1}
= \PP_k^{-1} + \big((B\!\otimes\! I_p)\Phi_k\big)\big((B\!\otimes\! I_p)\Phi_k\big)^{\!\top},
\]
we obtain
\begin{align}
&V_{k+1}
\le \widetilde{\Theta}_k^{\top}\bar \PP_{k+1}^{-1}\widetilde{\Theta}_k
+ 2\,\widetilde{\Theta}_k^{\top}\bar \PP_{k+1}^{-1}((B \bfa_k)\!\otimes\! I_p)\PP_k\Phi_k\Gamma_k\nonumber\\
&+ 2\,\widetilde{\Theta}_k^{\top}\bar \PP_{k+1}^{-1}((B \bfa_k)\!\otimes\! I_p)\PP_k\Phi_k\Xi_{k+1}\label{eq:V-expansion}\\
&+ \Gamma_k^{\top}\Phi_k^{\top}\PP_k((B \bfa_k)\!\otimes\! I_p)\bar \PP_{k+1}^{-1}((B \bfa_k)\!\otimes\! I_p)\PP_k\Phi_k\Gamma_k\nonumber\\
&+ \Xi_{k+1}^{\top}\Phi_k^{\top}\PP_k((B \bfa_k)\!\otimes\! I_p)\bar \PP_{k+1}^{-1}((B \bfa_k)\!\otimes\! I_p)\PP_k\Phi_k\Xi_{k+1}\nonumber\\
&+ 2\,\Gamma_k^{\top}\Phi_k^{\top}\PP_k((B \bfa_k)\!\otimes\! I_p)\bar \PP_{k+1}^{-1}((B \bfa_k)\!\otimes\! I_p)\PP_k\Phi_k\Xi_{k+1}.
\nonumber
\end{align}
First, since $\bar \PP_{k+1}^{-1} = \PP_k^{-1} + ((B\!\otimes\! I_p)\Phi_k)((B\!\otimes\! I_p)\Phi_k)^{\!\top}$,
\begin{equation}
\begin{aligned}
\widetilde{\Theta}_k^{\top}\bar \PP_{k+1}^{-1}\widetilde{\Theta}_k
&= 
V_k + \bar{\beta}^{2}\,\widetilde{\Theta}_{k}^{\top}\Phi_k\Phi_k^{\top}\widetilde{\Theta}_{k}.
\label{eq:first-term}
\end{aligned}
\end{equation}

Second, for the $\Gamma_k$-cross term, recall $B=\operatorname{diag}\{\beta_1,\dots,\beta_n\}$ and
$a_{k,i}=(1+\beta_i^{2}\varphi_{k,i}^{\top}P_{k,i}\varphi_{k,i})^{-1}$.
For $i\in[n]$, noting that
$
\bar P_{k+1,i}^{-1}=P_{k,i}^{-1}+\beta_i^{2}\varphi_{k,i}\varphi_{k,i}^{\top}
$, we thus have $
\bar P_{k+1,i}^{-1}(\beta_i a_{k,i})P_{k,i}$\\$\varphi_{k,i}
=\beta_i\,\varphi_{k,i},
$
so that, after block lifting,
\begin{equation}
\bar \PP_{k+1}^{-1}\big((B \bfa_k)\!\otimes\! I_p\big)\PP_k\Phi_k
=\ (B\!\otimes\! I_p)\,\Phi_k .
\label{eq:B-identity}
\end{equation}
Hence
\begin{align}
2\widetilde{\Theta}_k^{\top}\bar \PP_{k+1}^{-1}\big((B \bfa_k)\!\otimes\! I_p\big)\PP_k\Phi_k\,\Xi_k
\!=\!2\widetilde{\Theta}_k^{\top}(B\!\otimes\! I_p)\Phi_k\,\Xi_{k+1},\label{eq:Gamma-cross-Xi}
\end{align}
and
\begin{align}
&2\,\widetilde{\Theta}_k^{\top}\bar \PP_{k+1}^{-1}\big((B \bfa_k)\!\otimes\! I_p\big)\PP_k\Phi_k\,\Gamma_k
=2\,\widetilde{\Theta}_k^{\top}(B\!\otimes\! I_p)\Phi_k\,\Gamma_k \nonumber\\
=&2\sum_{i=1}^n \beta_i\,(\varphi_{k,i}^{\top}\widetilde{\theta}_{k,i})\,\gamma_{k,i}.\label{eq:Gamma-cross-B}
\end{align}
By the mean-value theorem, there exists
$\xi_{k,i}$ between $C-\varphi_{k,i}^{\top}\theta_{k,i}$ and $C-\varphi_{k,i}^{\top}\theta$
such that
$$
\begin{aligned}
\gamma_{k,i}
&=\big(1-(p_i+q_i)\big)\!\left[F_{k,i}\!\big(C-\varphi_{k,i}^{\top}\theta_{k,i}\big)-F_{k,i}\!\big(C-\varphi_{k,i}^{\top}\theta\big)\right]\\
&=-d_{k,i}\,\varphi_{k,i}^{\top}\widetilde{\theta}_{k,i},
\quad d_{k,i}:=\big(1-(p_i+q_i)\big)\,f_{k,i}(\xi_{k,i}).
\end{aligned}
$$
With the choice $\beta_i=\operatorname{sign}\!\big(1-(p_i+q_i)\big)\,|1-(p_i+q_i)|\,f_{\min}$ and
Assumption 3 that $f_{k,i}(\cdot)\ge f_{\min}$ on the working interval, we have
$\beta_i d_{k,i}\ge \beta_i^2$.
Substituting into \eqref{eq:Gamma-cross-B} gives
\begin{align}
&2\,\widetilde{\Theta}_k^{\top}(B\otimes I_p)\Phi_k\,\Gamma_k
=-2\sum_{i=1}^n \beta_i d_{k,i}\,(\varphi_{k,i}^{\top}\widetilde{\theta}_{k,i})^{2}\nonumber\\
&\le-2\sum_{i=1}^n \beta_i^{2}\,(\varphi_{k,i}^{\top}\widetilde{\theta}_{k,i})^{2}=-2\,\bar{\beta}^{2}\,\widetilde{\Theta}_k^{\top}\Phi_k\Phi_k^{\top}\widetilde{\Theta}_k .
\label{eq:Gamma-lb-B}
\end{align}
Third, for the quadratic terms, for each $i\in[n]$, a direct calculation yields
\[
\varphi_{k,i}^{\top}P_{k,i}\,(\beta_i a_{k,i})\,\bar P_{k+1,i}^{-1}\,(\beta_i a_{k,i})\,P_{k,i}\varphi_{k,i}
=\beta_i^{2}a_{k,i}\,\varphi_{k,i}^{\top}P_{k,i}\varphi_{k,i}.
\]
By block lifting, with $B:=\diag\{\beta_1,\dots,\beta_n\}$ and $a_k:=\diag\{a_{k,1},\dots,a_{k,n}\}$, 
we have
\begin{align}
&\Phi_k^{\top}\PP_k\big((B \bfa_k)\otimes I_p\big)\,\bar \PP_{k+1}^{-1}\,\big((B \bfa_k)\otimes I_p\big)\PP_k\Phi_k\nonumber\\
=&(B^{2}\bfa_k)\Phi_k^{\top}\PP_k\Phi_k.
\label{eq:key-identity-B}
\end{align}
Consequently, we obtain
\begin{align}
&\Gamma_k^{\top}\Phi_k^{\top}\PP_k\big((B \bfa_k)\!\otimes\! I_p\big)\bar \PP_{k+1}^{-1}\big((B \bfa_k)\!\otimes\! I_p\big)\PP_k\Phi_k\Gamma_k\nonumber\\
=&\Gamma_k^{\top}\,(B^{2}\bfa_k)\,\Phi_k^{\top}\PP_k\Phi_k\,\Gamma_k\label{eq:Gamma-quad-B},
\end{align}
\begin{align}
&\Xi_{k+1}^{\top}\Phi_k^{\top}\PP_k\big((B \bfa_k)\otimes I_p\big)\bar \PP_{k+1}^{-1}\big((B \bfa_k)\!\otimes I_p\big)\PP_k\Phi_k\Xi_{k+1}\nonumber\\
=&\Xi_{k+1}^{\top}\,(B^{2}\bfa_k)\,\Phi_k^{\top}\PP_k\Phi_k\,\Xi_{k+1},
\label{eq:Xi-quad-B}
\end{align}
and
\begin{align}
&2\,\Gamma_k^{\top}\,\Phi_k^{\top}\PP_k\big((B \bfa_k)\!\otimes\! I_p\big)\,
\bar \PP_{k+1}^{-1}\,\big((B \bfa_k)\!\otimes\! I_p\big)\,\PP_k\Phi_k\,\Xi_{k+1}\nonumber\\
&\qquad=\;2\,\Gamma_k^{\top}\,(B^{2}\bfa_k)\,\Phi_k^{\top}\PP_k\Phi_k\,\Xi_{k+1}.
\label{eq:Gamma-Xi-cross-B}
\end{align}

Summing \eqref{eq:V-expansion} for $k=0,1,\dots,t$ and using  \eqref{eq:first-term}, \eqref{eq:Gamma-cross-Xi}, and \eqref{eq:Gamma-lb-B}-\eqref{eq:Gamma-Xi-cross-B}, we have
\begin{align}
&V_{t+1} + \bar{\beta}^{2}\sum_{k=0}^{t}\widetilde{\Theta}_{k}^{\top}\Phi_k\Phi_k^{\top}\widetilde{\Theta}_{k}\!\le \!\sum_{k=0}^{t}\Gamma_k^{\top}(B^{2}\bfa_k )\,\Phi_k^{\top}\PP_k\Phi_k\,\Gamma_k +\nonumber\\
&\sum_{k=0}^{t}\Xi_{k+1}^{\top}(B^{2}\bfa_k)\Phi_k^{\top}\PP_k\Phi_k\Xi_{k+1}\!+\!2\sum_{k=0}^{t}\widetilde{\Theta}_k^{\top}(B\!\otimes\! I_p)\Phi_k\Xi_{k+1}\nonumber\\
&+2\sum_{k=0}^{t}\Gamma_k^{\top}(B^{2}\bfa_k )\,\Phi_k^{\top}\PP_k\Phi_k\,\Xi_{k+1}+V_0.\label{eq:master-sum-B}
\end{align}
We estimate the four sums on the right-hand side.

\noindent\emph{(i) The $\Gamma$–quadratic term.}
Since $|\gamma_{k,i}|\le 1$ and $\|B^2\|=\bar{\beta}^2$,
\begin{align}
\Gamma_k^{\top}(B^{2}\bfa_k)\,\Phi_k^{\top}\PP_k\Phi_k\,\Gamma_k
&\le \|\Gamma_k\|^2\,\|(B^{2}\bfa_k)\,\Phi_k^{\top}\PP_k\Phi_k\|\nonumber\\
&\le n\,\bar{\beta}^{2}\,\lambda_{\max}\!\big(\bfa_k\,\Phi_k^{\top}\PP_k\Phi_k\big).
\label{eq:Gamma-quad-est-B}
\end{align}

\noindent\emph{(ii) The $\Xi$–quadratic term.}
Using $|\varepsilon_{k+1,i}|\le 1$ a.s.\ so that $\sum_{i=1}^n \varepsilon_{k+1,i}^2\le n$, and thus
\begin{align}
\Xi_{k+1}^{\top}(\!B^{2}\bfa_k\!)\Phi_k^{\top}\!\PP_k\!\Phi_k\,\Xi_{k+1}
&\!\le\! \lambda_{\max}\!\big((B^{2}\bfa_k)\Phi_k^{\top}\PP_k\Phi_k\big)\sum_{i=1}^n \varepsilon_{k+1,i}^{2}\nonumber\\
&\!\le\! n\bar{\beta}^2\,\lambda_{\max}\!\big(a_k\,\Phi_k^{\top}\PP_k\Phi_k\big).
\label{eq:Xi-quad-est-B}
\end{align}

\noindent\emph{(iii) The $\widetilde\Theta$–$\Xi$ term.}
Since $(\Xi_{k+1})_{k\ge0}$ is a bounded martingale difference array and
$\widetilde{\Theta}_k^{\top}(B\!\otimes\! I_p)\Phi_k\in\mathcal{F}_k$, for any $\delta\in(0,\tfrac12)$
the martingale estimate theorem (e.g. Theorem 2.8 in \cite{chen2012identification}) yields
\begin{align}
&\sum_{k=0}^{t}\widetilde{\Theta}_k^{\top}(B\otimes I_p)\Phi_k\,\Xi_{k+1}\nonumber\\
=&O\!\left(\Big[\sum_{k=0}^{t}\|\widetilde{\Theta}_k^{\top}(B\!\otimes\! I_p)\Phi_k\Phi_k^{\top}(B\!\otimes\! I_p)\widetilde{\Theta}_k\|\Big]^{\frac12+\delta}\right)\nonumber\\
=&o\!\left(\sum_{k=0}^{t}\widetilde{\Theta}_{k}^{\top}\Phi_k\Phi_k^{\top}\widetilde{\Theta}_{k}\right)+O(1),
\label{eq:Theta-Xi-mdb-B}
\end{align}
where the last step uses the boundedness of $B$.

\noindent\emph{(iv) The $\Gamma$–$\Xi$ term.}
Similarly,
\begin{align}
&\sum_{k=0}^{t}\Gamma_k^{\top}(B^{2}\bfa_k)\,\Phi_k^{\top}\PP_k\Phi_k\,\Xi_{k+1}\nonumber\\
=&O\!\left(\Big[\sum_{k=0}^{t}\big\|(B^{2}\bfa_k)\,\Phi_k^{\top}\PP_k\Phi_k\,\Phi_k^{\top}\PP_k\Phi_k\,(B^{2}\bfa_k)\big\|\Big]^{\frac12+\delta}\right)\nonumber\\
\le& O\!\left(\Big[\sum_{k=0}^{t}\bar{\beta}^{4}\,\|a_k\,\Phi_k^{\top}\PP_k\Phi_k\|^{2}\Big]^{\frac12+\delta}\right)\nonumber\\
=&o\!\left(\sum_{k=0}^{t}\|a_k\,\Phi_k^{\top}\PP_k\Phi_k\|\right)+O(1),
\label{eq:Gamma-Xi-mdb-B}
\end{align}
again for any $\delta\in(0,\tfrac12)$.

Let $U_k:=(B\!\otimes\! I_p)\Phi_k\in\mathbb{R}^{(np)\times n}$ and note from
\eqref{eq:block_ATC} that
$
\bar \PP_{k+1}^{-1}=\PP_k^{-1}+U_kU_k^{\top}.
$
By the matrix determinant lemma,
\[
\big|\bar \PP_{k+1}^{-1}\big|
=\big|\PP_k^{-1}\big|\;\big|I_n+U_k^{\top}\PP_kU_k\big|.
\]
Since $\PP_k=\mathrm{diag}\{P_{k,1},\dots,P_{k,n}\}$ and $U_k$ has columns
$\{\beta_i\varphi_{k,i}\}_{i=1}^n$, we have
$
U_k^{\top}\PP_kU_k
=\mathrm{diag}\!\big\{\beta_i^2\,\varphi_{k,i}^{\top}P_{k,i}\varphi_{k,i}\big\}_{i=1}^n
= B^2\,\Phi_k^{\top}\PP_k\Phi_k,
$
whence
$
\frac{\big|\bar \PP_{k+1}^{-1}\big|}{\big|\PP_k^{-1}\big|}
=\big|I_n+B^2\,\Phi_k^{\top}\PP_k\Phi_k\big|.
$
Recall that $\bfa_k=(I_n+B^2\,\Phi_k^{\top}\PP_k\Phi_k)^{-1}$ and therefore
\[
\big|a_k\big|
=\frac{\big|\PP_k^{-1}\big|}{\big|\bar \PP_{k+1}^{-1}\big|}
=1-\Big(1-\frac{\big|\PP_k^{-1}\big|}{\big|\bar \PP_{k+1}^{-1}\big|}\Big).
\]
Since $\bfa_k$ is diagonal and $0\preceq \bfa_k\Phi_k^{\top}\PP_k\Phi_k\preceq I_n$, we have
$$
\begin{aligned}
&\lambda_{\max}\!\big(\bfa_k\,\Phi_k^{\top}\PP_k\Phi_k\big)
=\max_{i\in[n]}\frac{\varphi_{k,i}^{\top}P_{k,i}\varphi_{k,i}}{1+\beta_i^2\,\varphi_{k,i}^{\top}P_{k,i}\varphi_{k,i}}\\
\le&\frac{1}{\bar{\beta}^{2}}\,
\max_{i\in[n]}\frac{\beta_i^2\,\varphi_{k,i}^{\top}P_{k,i}\varphi_{k,i}}{1+\beta_i^2\,\varphi_{k,i}^{\top}P_{k,i}\varphi_{k,i}}\\
\le&\frac{1}{\bar{\beta}^{2}}\!\left(1-\frac{1}{\big|I_n+B^2\,\Phi_k^{\top}\PP_k\Phi_k\big|}\right).
\end{aligned}
$$
Using the formula above and Lemma~\ref{lem:P-rel} (i.e., $\big|\bar \PP_{k+1}^{-1}\big|\le \big|\PP_{k+1}^{-1}\big|$), it follows
$$
\begin{aligned}
\lambda_{\max}\!\big(\bfa_k\,\Phi_k^{\top}\PP_k\Phi_k\big)
\le\frac{1}{\bar{\beta}^{2}}
\left(\frac{\big|\PP_{k+1}^{-1}\big|-\big|\PP_k^{-1}\big|}{\big|\PP_{k+1}^{-1}\big|}\right).
\end{aligned}
$$
Summing and comparing with the integral of $x^{-1}$,

\begin{align}
&\sum_{k=0}^{t}\lambda_{\max}\!\big(\bfa_k\,\Phi_k^{\top}\PP_k\Phi_k\big)
\le\frac{1}{\bar{\beta}^{2}}\sum_{k=0}^{t}\int_{|\PP_k^{-1}|}^{|\PP_{k+1}^{-1}|}\!\frac{dx}{x}\nonumber\\
=&\frac{1}{\bar{\beta}^{2}}\Big(\log|\PP_{t+1}^{-1}|-\log|\PP_{0}^{-1}|\Big),\label{eq:logdet-bound-B}
\end{align}

Combining \eqref{eq:Gamma-quad-est-B}–\eqref{eq:Gamma-Xi-mdb-B} with
\eqref{eq:logdet-bound-B}, 
substituting these bounds into \eqref{eq:master-sum-B} and absorbing the
lower-order items, we have
\begin{equation}\label{bound1}
\begin{aligned}
V_{t+1}
+\bar{\beta}^2\sum_{k=0}^{t}\widetilde{\Theta}_{k}^{\top}\Phi_k\Phi_k^{\top}\widetilde{\Theta}_{k}=&O (\log|\PP_{t+1}^{-1}|),\ \mathrm{a.s.}
\end{aligned}
\end{equation}
From the combination step 
$\operatorname{vec}(\PP_{t+1}^{-1})=\mathcal{A}\operatorname{vec}(\bar \PP_{t+1}^{-1})$) we have, for any $i\in[n]$,
\[
\PP_{t+1,i}^{-1}
=\sum_{j=1}^{n} a_{ij}\,\bar \PP_{t+1,j}^{-1}
=\sum_{j=1}^{n} a_{ij}\Big(\PP_{t,j}^{-1}+\beta_j^2\,\varphi_{t,j}\varphi_{t,j}^{\top}\Big).
\]
Using $\lambda_{\max}(\sum_j a_{ij}X_j)\le \sum_j a_{ij}\lambda_{\max}(X_j)$, the row-stochasticity
$\sum_j a_{ij}=1$, and $\lambda_{\max}(\varphi\varphi^{\top})=\|\varphi\|^2$, we obtain
\[
\max_{i\in[n]}\lambda_{\max}(P_{t+1,i}^{-1})
\;\le\;
\max_{j\in[n]}\lambda_{\max}(P_{t,j}^{-1})
\;+\;\bar{\beta}^2\sum_{j=1}^{n}\|\varphi_{t,j}\|^{2}.
\]
Iterating this bound over $t,t-1,\dots,0$ yields
\[
\max_{i\in[n]}\lambda_{\max}(P_{t+1,i}^{-1})
\;\le\;
\max_{i\in[n]}\lambda_{\max}(P_{0,i}^{-1})
+\bar{\beta}^2\sum_{j=1}^{n}\sum_{k=0}^{t}\|\varphi_{k,j}\|^{2}.
\]
Since $\PP_{t+1}=\operatorname{diag}\{P_{t+1,1},\dots,P_{t+1,n}\}\in\mathbb{S}_{++}^{np}$,
\begin{align}
&\log|\PP_{t+1}^{-1}|
\le
np\log\Big(\max_{i\in[n]}\lambda_{\max}(P_{t+1,i}^{-1})\Big)\nonumber\\
\le&
np\;\log\!\Big(c_0+c_1\sum_{j=1}^{n}\sum_{k=0}^{t}\|\varphi_{k,j}\|^{2}\Big),\label{bound2}
\end{align}
for constants $c_0:=\max_{i\in[n]}\lambda_{\max}(P_{0,i}^{-1})$ and $c_1:=\bar{\beta}^2$. Hence, with
$
r_t:=\max_{i\in[n]}\lambda_{\max}(P_{0,i})+\sum_{i=1}^{n}\sum_{k=0}^{t}\|\varphi_{k,i}\|^{2},
$
combining (\ref{bound1})-(\ref{bound2}) gives (\ref{thm1})-(\ref{thm2}). \hfill$\square$

Assume the disturbance is integrable so that $\mathbb{E}[\,w_{k+1,i}\mid\mathcal{F}_k]$ exists.
For any $i\in[n]$ and $k\ge0$, the best mean-square predictor is
$
\mathbb{E}\!\left[y_{k+1,i}\mid\mathcal{F}_k\right]
=\varphi_{k,i}^{\top}\theta+\mathbb{E}\!\left[w_{k+1,i}\mid\mathcal{F}_k\right].
$
Replacing $\theta$ with its online estimate $\theta_{k,i}$ yields the adaptive predictor
$
\widehat{y}_{k+1,i}
=\varphi_{k,i}^{\top}\theta_{k,i}+\mathbb{E}\!\left[w_{k+1,i}\mid\mathcal{F}_k\right].
$
The instantaneous regret is
$
R_{k,i}
=\Big(\mathbb{E}[y_{k+1,i}\mid\mathcal{F}_k]-\widehat{y}_{k+1,i}\Big)^2
=\big(\varphi_{k,i}^{\top}\widetilde{\theta}_{k,i}\big)^2.
$
Consequently, the accumulated regret satisfies
\[
\sum_{i=1}^{n}\sum_{k=0}^{t}R_{k,i}
=\sum_{k=0}^{t}\widetilde{\Theta}_{k}^{\top}\Phi_{k}\Phi_{k}^{\top}\widetilde{\Theta}_{k}.
\]

With these preliminaries in place, we now present accumulated regret and the parameter error bound.
\begin{theorem}[Accumulated regret]\label{thm:regret}
Consider the Algorithm \ref{algo1} with measurement-side binary tampering \eqref{eq:flip}. 
Under Assumptions~1-4 and $p_i+q_i\neq 1$ for all $i\in[n]$, the sample paths satisfy, as $t\to\infty$,
\[
\sum_{i=1}^{n}\sum_{k=0}^{t} R_{k,i}
= O\!\big(\log r_t\big),\quad \text{a.s.},
\]
where $r_t$ is as in Theorem \ref{thm:growth}.
\end{theorem}

{\bf Proof.}
Theorem~\ref{thm:growth} directly yields the claim.\hfill$\square$

\begin{theorem}[Parameter convergence]\label{thm:param-rate}
Consider the Algorithm \ref{algo1} with measurement-side binary tampering \eqref{eq:flip}. Suppose Assumptions~1-4 hold. Then, as $t\to\infty$,
\[
\|\widetilde{\Theta}_{t+1}\|^{2}
=O\!\left(\frac{\log r_t}{\lambda_{\min}^{n,t}}\right)\quad\text{a.s.},
\]
where $r_t$ is as in Theorem \ref{thm:growth} and
\begin{equation}\label{eq:def-lmin-nt}
\lambda_{\min}^{n,t}
:=\lambda_{\min}\!\left(\sum_{j=1}^{n}P_{0,j}^{-1}
+\sum_{j=1}^{n}\sum_{k=0}^{\,t-D(G)+1}\varphi_{k,j}\varphi_{k,j}^{\top}\right).
\end{equation}
\end{theorem}

\noindent{\bf Proof.}
From the block recursion \eqref{eq:block_ATC},
$$
\begin{aligned}
&\operatorname{vec}\!\big(\PP_{t+1}^{-1}\big)
=\mathcal A^{t+1}\operatorname{vec}\!\big(\PP_{0}^{-1}\big)\\
&+\sum_{k=0}^{t}{\mathcal A}^{\,t-k+1}\operatorname{vec}\!\big(((B\!\otimes\! I_p)\Phi_k)\big((B\!\otimes\! I_p)\Phi_k\big)^{\!\top}\big),
\end{aligned}
$$
so for each $i\in [n]$,
$
P_{t+1,i}^{-1}
=\sum_{j=1}^{n} (A^{t+1})_{ij} P_{0,j}^{-1}
+\sum_{j=1}^{n}\sum_{k=0}^{t} (A^{t-k+1})_{ij}\,\beta_j^{2}\,\varphi_{k,j}\varphi_{k,j}^{\top}.
$
Because $G$ is connected and $A$ is symmetric, there exists $\underline a>0$ (depending only on $A$ and $D(G)$) such that
$\min_{i,j\in[n]}(A^{d})_{ij}\ge \underline a$ for all $d\ge D(G)$. Hence, for $t\ge D(G)-1$,
\[
P_{t+1,i}^{-1}\ \succeq\ 
\underline a\sum_{j=1}^{n}P_{0,j}^{-1}
+\underline a\,\bar{\beta}^{2}\sum_{j=1}^{n}\sum_{k=0}^{t-D(G)+1}\varphi_{k,j}\varphi_{k,j}^{\top},
\]
and therefore
$
\lambda_{\min}\!\big(\PP_{t+1}^{-1}\big)\ \ge\ c_0\,\lambda_{\min}^{n,t}$ with
$c_0:=\underline a\,\min\{1,\bar{\beta}\}.
$
Thus by Theorem~\ref{thm:growth},
\[
\|\widetilde{\Theta}_{t+1}\|^{2}
\le \frac{1}{\lambda_{\min}(\PP_{t+1}^{-1})}\,
\widetilde{\Theta}_{t+1}^{\top}\PP_{t+1}^{-1}\widetilde{\Theta}_{t+1}=
O\!\left(\frac{\log r_t}{\lambda_{\min}^{n,t}}\right)\ \text{a.s.}
\]
This completes the proof. \hfill$\square$

\section{Simulation}
We verify parameter convergence under binary tampering on a sparse connected network. Let \(n=6\), \(p=6\), threshold \(C=1\), horizon \(T=6000\). The true parameter is \(\theta=[3,1,-2.5, -0.5, 2, -1.5]^\top\) and the projection set is \(\Omega=[-4,4]^p\). Noises are i.i.d.\ \(w_{k,i}\sim\mathcal N(0,\sigma_w^2)\) with \(\sigma_w=8\). For each node \(i\), \(\theta_{0,i}\) is a constant vector whose entries are chosen from \(\{-3,-2,-1,1,2,3\}\), and \(P_{0,i}=10\,I_p\). Regressors are deterministic, bounded, and block-sparse: each node excites one coordinate \(j(i)=1+((i-1)\bmod p)\) via \(\varphi_{k,i}=\sigma_{j}(2-\rho_{j}^{-k})e_{j}\) with \(\rho_{j}=j+1\) and \(\sigma_{j}=(-1)^{j-1}\), so that \(\|\varphi_{k,i}\|\le 2\), each node is rank one, and the network is jointly exciting. Tampering uses the flip probabilities \((p_i,q_i)\) over $\{(0.15,0.10),(0.90,0.40),(0.90,0.80),(0.10,0.10),$\\$(0.15,0.10),(0.90,0.40)\}$. In Algorithm \ref{algo1} we use these true \((p_i,q_i)\), while in the tampering-unaware recursion we set \(p_i=q_i=0\) in the update. Communication employs a symmetric, doubly-stochastic tridiagonal matrix
$A\in\mathbb{R}^{6\times 6}$, where
each interior node assigns weight $1/3$ to itself and each neighbor,
and each boundary node assigns weight $2/3$ to itself and $1/3$ to its
single neighbor.

We run Algorithm~\ref{algo1} 
and record \(\mathrm{MSE}_i(k)=\|\hat\theta_{k,i}-\theta\|^2\). Fig.~\ref{fig:thetaTraj} overlays all node-coordinate trajectories under Algorithm \ref{algo1} and shows convergence of every coordinate to the corresponding entry of \(\theta\). Fig.~\ref{fig:meanMSE_vs_I} plots the node-averaged \(\mathrm{MSE}_i(k)\) for Algorithm~\ref{algo1} and for the non-cooperative baseline with \(A=I_n\); only the collaborative method converges to zero, whereas the non-cooperative case does not. Fig.~\ref{fig:meanMSE_unaware} contrasts Algorithm \ref{algo1} with a tampering-unaware recursion that sets \(p_i=q_i=0\): the former achieves vanishing mean MSE and converges to 0, whereas the latter remains biased and fails to converge, showing that modeling flips is essential for consistency.

\begin{figure}[t]
  \centering
  \includegraphics[width=0.8\linewidth]{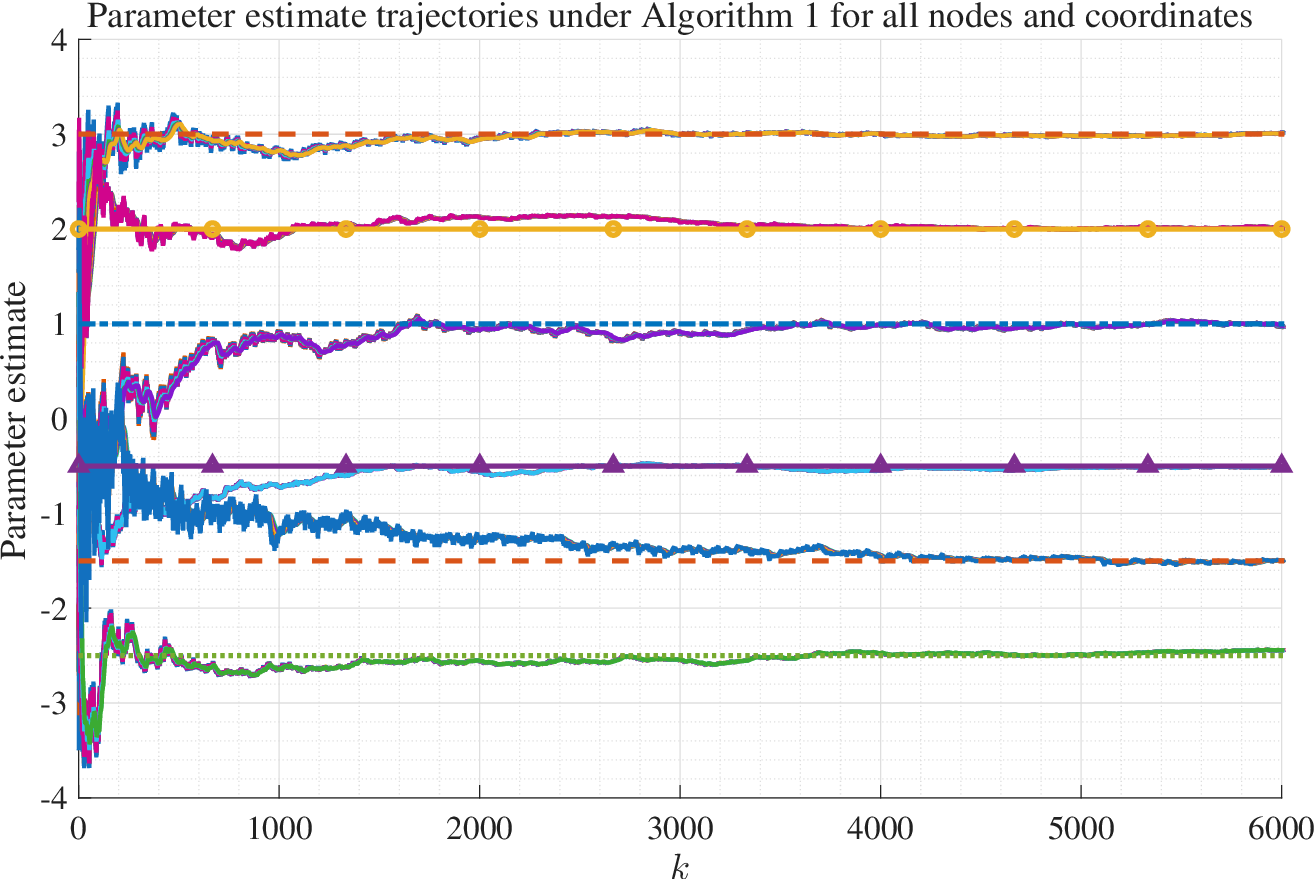}
  \caption{Parameter estimate trajectories under Algorithm~\ref{algo1} for all nodes and all coordinates.}
  \label{fig:thetaTraj}
\end{figure}
\begin{figure}[t]
  \centering
  \includegraphics[width=0.8\linewidth]{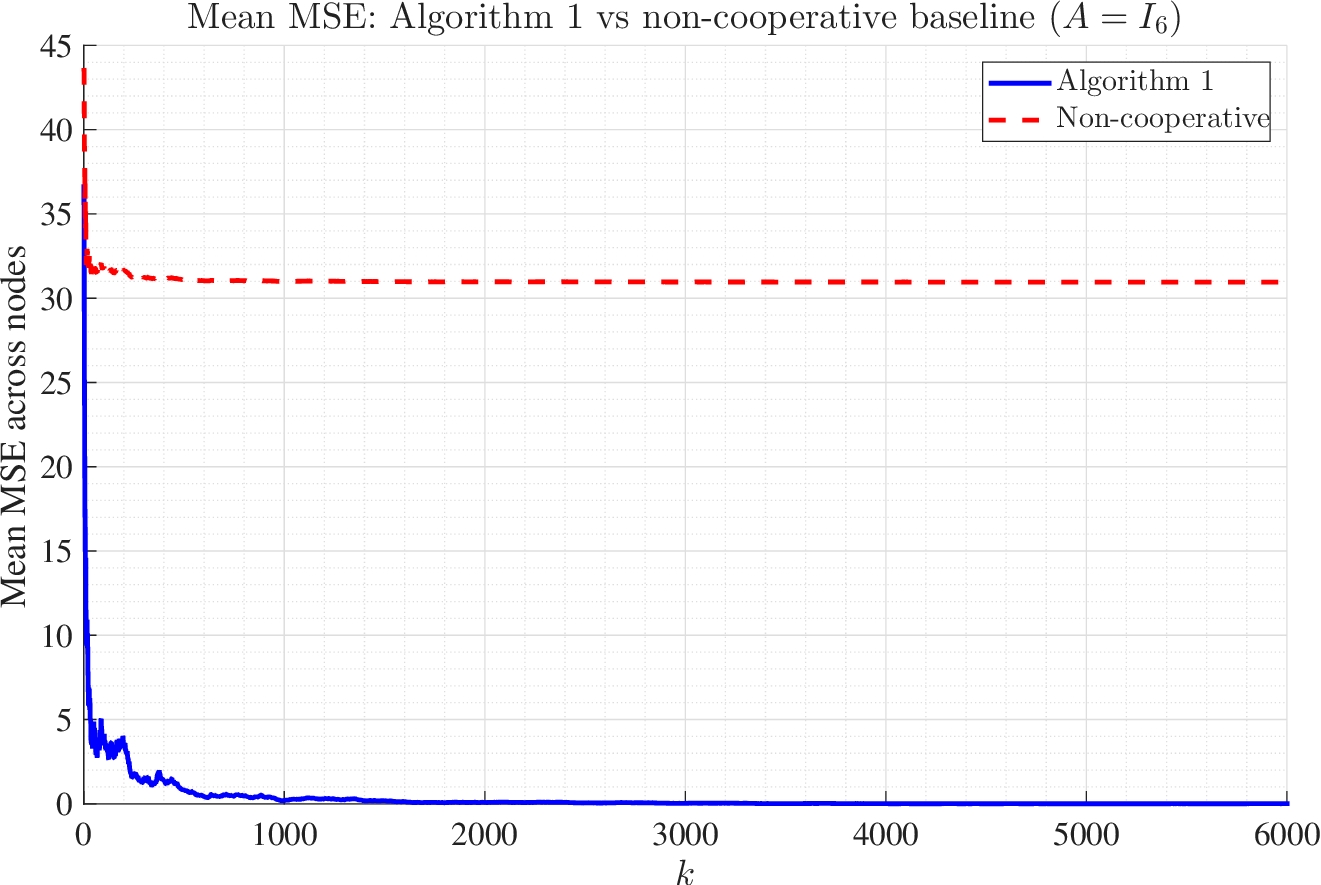}
  \caption{Mean MSE across nodes for Algorithm~\ref{algo1} versus the non-cooperative baseline \(A=I_n\).}
  \label{fig:meanMSE_vs_I}
\end{figure}

\begin{figure}[t]
  \centering
  \includegraphics[width=0.8\linewidth]{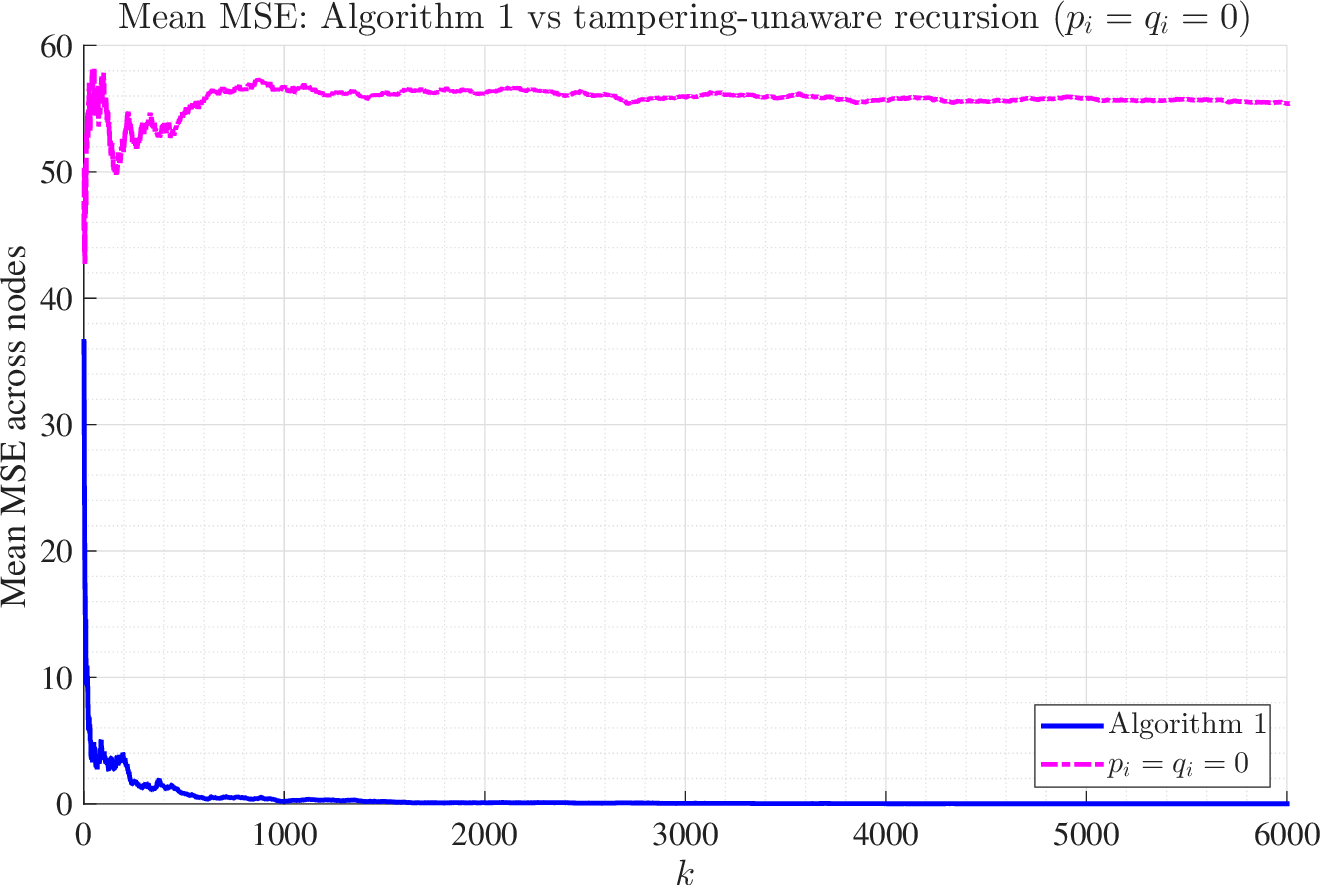}
  \caption{Mean MSE across nodes for Algorithm~\ref{algo1} versus a tampering-unaware recursion that sets \(p_i=q_i=0\).}
  \label{fig:meanMSE_unaware}
\end{figure}

\section{Conclusion}

We developed a convergence theory for a class of distributed recursive projection algorithms operating with binary observations subject to measurement-side tampering. Under binary observations and flip attacks, we proved that the accumulated regret of the adaptive predictor grows logarithmically. Moreover, under a cooperative excitation condition, we established almost-sure convergence of all nodewise estimates to the true parameter. Our analysis does not require independence, stationarity, or Gaussian assumptions on the regressors or noises. The cooperative excitation condition further shows that the proposed distributed algorithm can complete the estimation task collectively even when no single sensor is sufficiently exciting on its own, thereby extending the weakest known excitation requirements from classical least-squares to the binary–tampering distributed setting.

Future research directions include considering identification in which the flip probabilities \((p_i,q_i)\) are unknown or time-varying and estimated online together with \(\theta\). Another direction is to give conditions for how many attacked nodes can be tolerated, and cases where correct identification is impossible.

\bibliographystyle{siamplain}
\bibliography{ifacconf}

@book{ljung1995system,
  title={System identification toolbox: User's guide},
  author={Ljung, Lennart},
  year={1995},
  publisher={Citeseer}
}

@article{cattivelli2009diffusion,
  title={Diffusion LMS strategies for distributed estimation},
  author={Cattivelli, Federico S and Sayed, Ali H},
  journal={IEEE Transactions on Signal Processing},
  volume={58},
  number={3},
  pages={1035--1048},
  year={2009},
  publisher={IEEE}
}

@article{chen2012diffusion,
  title={Diffusion adaptation strategies for distributed optimization and learning over networks},
  author={Chen, Jianshu and Sayed, Ali H},
  journal={IEEE Transactions on Signal Processing},
  volume={60},
  number={8},
  pages={4289--4305},
  year={2012},
  publisher={IEEE}
}

@article{cattivelli2010diffusion,
  title={Diffusion strategies for distributed Kalman filtering and smoothing},
  author={Cattivelli, Federico S and Sayed, Ali H},
  journal={IEEE Transactions on Automatic Control},
  volume={55},
  number={9},
  pages={2069--2084},
  year={2010},
  publisher={IEEE}
}

@article{alghunaim2024local,
  title={Local exact-diffusion for decentralized optimization and learning},
  author={Alghunaim, Sulaiman A},
  journal={IEEE Transactions on Automatic Control},
  volume={69},
  number={11},
  pages={7371--7386},
  year={2024},
  publisher={IEEE}
}

@article{xie2018analysis,
  title={Analysis of distributed adaptive filters based on diffusion strategies over sensor networks},
  author={Xie, Siyu and Guo, Lei},
  journal={IEEE Transactions on Automatic Control},
  volume={63},
  number={11},
  pages={3643--3658},
  year={2018},
  publisher={IEEE}
}

@article{xie2020convergence,
  title={Convergence of a distributed least squares},
  author={Xie, Siyu and Zhang, Yaqi and Guo, Lei},
  journal={IEEE Transactions on Automatic Control},
  volume={66},
  number={10},
  pages={4952--4959},
  year={2020},
  publisher={IEEE}
}

@article{gan2023distributed,
  title={Distributed sparse identification for stochastic dynamic systems under cooperative non-persistent excitation condition},
  author={Gan, Die and Liu, Zhixin},
  journal={Automatica},
  volume={151},
  pages={110958},
  year={2023},
  publisher={Elsevier}
}

@article{kar2012distributed,
  title={Distributed parameter estimation in sensor networks: Nonlinear observation models and imperfect communication},
  author={Kar, Soummya and Moura, Jos{\'e} MF and Ramanan, Kavita},
  journal={IEEE Transactions on Information Theory},
  volume={58},
  number={6},
  pages={3575--3605},
  year={2012},
  publisher={IEEE}
}

@article{michelusi2022finite,
  title={Finite-bit quantization for distributed algorithms with linear convergence},
  author={Michelusi, Nicolo and Scutari, Gesualdo and Lee, Chang-Shen},
  journal={IEEE Transactions on Information Theory},
  volume={68},
  number={11},
  pages={7254--7280},
  year={2022},
  publisher={IEEE}
}

@article{zhu2018mitigating,
  title={Mitigating quantization effects on distributed sensor fusion: A least squares approach},
  author={Zhu, Shanying and Chen, Cailian and Xu, Jinming and Guan, Xinping and Xie, Lihua and Johansson, Karl Henrik},
  journal={IEEE Transactions on Signal Processing},
  volume={66},
  number={13},
  pages={3459--3474},
  year={2018},
  publisher={IEEE}
}

@article{iakovidou2022s,
  title={S-NEAR-DGD: A flexible distributed stochastic gradient method for inexact communication},
  author={Iakovidou, Charikleia and Wei, Ermin},
  journal={IEEE Transactions on Automatic Control},
  volume={68},
  number={2},
  pages={1281--1287},
  year={2022},
  publisher={IEEE}
}

@article{wang2021distributed,
  title={Distributed recursive projection identification with binary-valued observations},
  author={Wang, Ying and Zhao, Yanlong and Zhang, Ji-Feng},
  journal={Journal of Systems Science and Complexity},
  volume={34},
  number={5},
  pages={2048--2068},
  year={2021},
  publisher={Springer}
}

@article{lu2025consensus,
  title={Consensus of multi-agent systems under binary-valued measurements: An event-triggered coordination approach},
  author={Lu, Xiaodong and Wang, Ting and Zhao, Yanlong and Zhang, Ji-Feng},
  journal={Automatica},
  volume={176},
  pages={112255},
  year={2025},
  publisher={Elsevier}
}

@article{wang2026distributed,
  title={Distributed estimation with quantized measurements and communication over Markovian switching topologies},
  author={Wang, Ying and Guo, Jian and Zhao, Yanlong and Zhang, Ji-feng},
  journal={Automatica},
  volume={183},
  pages={112658},
  year={2026},
  publisher={Elsevier}
}

@inproceedings{yin2018byzantine,
  title={Byzantine-robust distributed learning: Towards optimal statistical rates},
  author={Yin, Dong and Chen, Yudong and Kannan, Ramchandran and Bartlett, Peter},
  booktitle={International Conference on Machine Learning},
  pages={5650--5659},
  year={2018},
  organization={Pmlr}
}

@article{pasqualetti2013attack,
  title={Attack detection and identification in cyber-physical systems},
  author={Pasqualetti, Fabio and D{\"o}rfler, Florian and Bullo, Francesco},
  journal={IEEE Transactions on Automatic Control},
  volume={58},
  number={11},
  pages={2715--2729},
  year={2013},
  publisher={IEEE}
}

@article{an2019distributed,
  title={Distributed secure state estimation for cyber--physical systems under sensor attacks},
  author={An, Liwei and Yang, Guang-Hong},
  journal={Automatica},
  volume={107},
  pages={526--538},
  year={2019},
  publisher={Elsevier}
}

@article{su2019finite,
  title={Finite-time guarantees for Byzantine-resilient distributed state estimation with noisy measurements},
  author={Su, Lili and Shahrampour, Shahin},
  journal={IEEE Transactions on Automatic Control},
  volume={65},
  number={9},
  pages={3758--3771},
  year={2019},
  publisher={IEEE}
}

@article{an2021byzantine,
  title={Byzantine-resilient distributed state estimation: A min-switching approach},
  author={An, Liwei and Yang, Guang-Hong},
  journal={Automatica},
  volume={129},
  pages={109664},
  year={2021},
  publisher={Elsevier}
}

@article{de2015input,
  title={Input-to-state stabilizing control under denial-of-service},
  author={De Persis, Claudio and Tesi, Pietro},
  journal={IEEE Transactions on Automatic Control},
  volume={60},
  number={11},
  pages={2930--2944},
  year={2015},
  publisher={IEEE}
}

@article{chen2018resilient,
  title={Resilient distributed estimation through adversary detection},
  author={Chen, Yuan and Kar, Soummya and Moura, Jose MF},
  journal={IEEE Transactions on Signal Processing},
  volume={66},
  number={9},
  pages={2455--2469},
  year={2018},
  publisher={IEEE}
}

@article{he2021secure,
  title={How to secure distributed filters under sensor attacks},
  author={He, Xingkang and Ren, Xiaoqiang and Sandberg, Henrik and Johansson, Karl Henrik},
  journal={IEEE Transactions on Automatic Control},
  volume={67},
  number={6},
  pages={2843--2856},
  year={2021},
  publisher={IEEE}
}

@inproceedings{fang2018robustifying,
  title={Robustifying the Kalman filter against measurement outliers: An innovation saturation mechanism},
  author={Fang, Huazhen and Haile, Mulugeta A and Wang, Yebin},
  booktitle={2018 IEEE Conference on Decision and Control (CDC)},
  pages={6390--6395},
  year={2018},
  organization={IEEE}
}

@article{fang2021robust,
  title={Robust extended Kalman filtering for systems with measurement outliers},
  author={Fang, Huazhen and Haile, Mulugeta A and Wang, Yebin},
  journal={IEEE Transactions on Control Systems Technology},
  volume={30},
  number={2},
  pages={795--802},
  year={2021},
  publisher={IEEE}
}

@article{yu2022robust,
  title={Robust resilient diffusion over multi-task networks against Byzantine attacks: Design, analysis and applications},
  author={Yu, Tao and de Lamare, Rodrigo C and Yu, Yi},
  journal={IEEE Transactions on Signal Processing},
  volume={70},
  pages={2826--2841},
  year={2022},
  publisher={IEEE}
}

@article{han2025masked,
  title={Masked Diffusion Strategy for Privacy-Preserving Distributed Learning},
  author={Han, Hongyu and Zhang, Sheng and Chen, Hongyang and Sayed, Ali H},
  journal={IEEE Transactions on Information Forensics and Security},
  year={2025},
  publisher={IEEE}
}

@article{wu2020zero,
  title={Zero-sum game-based optimal secure control under actuator attacks},
  author={Wu, Chengwei and Li, Xiaolei and Pan, Wei and Liu, Jianxing and Wu, Ligang},
  journal={IEEE Transactions on Automatic Control},
  volume={66},
  number={8},
  pages={3773--3780},
  year={2020},
  publisher={IEEE}
}

@article{paarporn2024strategically,
  title={Strategically revealing intentions in General Lotto games},
  author={Paarporn, Keith and Chandan, Rahul and Kovenock, Dan and Alizadeh, Mahnoosh and Marden, Jason R},
  journal={IEEE Transactions on Automatic Control},
  volume={69},
  number={8},
  pages={5396--5407},
  year={2024},
  publisher={IEEE}
}

@book{farrar2012structural,
  title={Structural health monitoring: a machine learning perspective},
  author={Farrar, Charles R and Worden, Keith},
  year={2012},
  publisher={John Wiley \& Sons}
}

@article{farrar2007introduction,
  title={An introduction to structural health monitoring},
  author={Farrar, Charles R and Worden, Keith},
  journal={Philosophical Transactions of the Royal Society A: Mathematical, Physical and Engineering Sciences},
  volume={365},
  number={1851},
  pages={303--315},
  year={2007},
  publisher={The Royal Society London}
}

@article{khraisat2019survey,
  title={Survey of intrusion detection systems: techniques, datasets and challenges},
  author={Khraisat, Ansam and Gondal, Iqbal and Vamplew, Peter and Kamruzzaman, Joarder},
  journal={Cybersecurity},
  volume={2},
  number={1},
  pages={1--22},
  year={2019},
  publisher={Springer}
}

@article{banerjee2008fault,
  title={Fault tolerant multiple event detection in a wireless sensor network},
  author={Banerjee, Torsha and Xie, Bin and Agrawal, Dharma P},
  journal={Journal of Parallel and Distributed Computing},
  volume={68},
  number={9},
  pages={1222--1234},
  year={2008},
  publisher={Elsevier}
}

@article{hager1989updating,
  title={Updating the inverse of a matrix},
  author={Hager, William W},
  journal={SIAM Review},
  volume={31},
  number={2},
  pages={221--239},
  year={1989},
  publisher={SIAM}
}

@book{chen2012identification,
  title={Identification and stochastic adaptive control},
  author={Chen, Han-Fu and Guo, Lei},
  year={2012},
  publisher={Springer Science \& Business Media}
}

@misc{guo2025recursivebinaryidentificationdata,
      title={Recursive Binary Identification under Data Tampering and Non-Persistent Excitation with Application to Emission Control}, 
      author={Jian Guo and Lihong Pei and Wenchao Xue and Yanlong Zhao and Ji-Feng Zhang},
      year={2025},
      eprint={2511.08629},
      archivePrefix={arXiv},
      primaryClass={eess.SY},
      url={https://arxiv.org/abs/2511.08629}, 
}

@article{guo2025state,
  title={State estimation with protecting exogenous inputs via Cram{\'e}r-Rao lower bound approach},
  author={Guo, Liping and Wang, Jimin and Zhao, Yanlong and Zhang, Ji-Feng},
  journal={arXiv preprint arXiv:2410.08756},
  year={2025}
}

@article{zhang2026dive,
  title={Dive into streaming: efficient identification of encrypted dynamic DASH video traffic},
  author={Zhang, Xiyuan and Xiong, Gang and Gou, Gaopeng and Li, Zhen and Gu, Zheyuan and Huang, Yiyang and Fang, Binxing},
  journal={Science China Information Sciences},
  volume={69},
  number={1},
  pages={1--21},
  year={2026},
  publisher={Springer}
}

@article{su2025event,
  title={Event-triggered leader-follower bipartite consensus control for nonlinear multi-agent systems under DoS attacks},
  author={Su, Wei and Mu, Chaoxu and Zhu, Song and Niu, Ben and Sun, Changyin},
  journal={Science China Information Sciences},
  volume={68},
  number={3},
  pages={132206},
  year={2025},
  publisher={Springer}
}

\end{document}